\documentclass[11pt, letterpaper,english]{amsart}
\usepackage[left=1in,right=1in,bottom=1in,top=1in]{geometry}

\usepackage{graphicx}
\usepackage{multicol,multirow}
\usepackage{amsmath,amssymb,amsfonts}
\usepackage{mathrsfs}
\usepackage{amsthm}
\usepackage{rotating}
\usepackage{appendix}
\usepackage[numbers]{natbib}
\usepackage{ifpdf}
\usepackage{xcolor}
\usepackage[colorlinks,allcolors=blue]{hyperref}
\usepackage{lipsum}
\newtheorem{theorem}{Theorem}[section]
\newtheorem{lemma}[theorem]{Lemma}
\newtheorem{corollary}[theorem]{Corollary}

\newtheorem{proposition}[theorem]{Proposition}
\newtheorem{conjecture}[theorem]{Conjecture}

\newtheorem{remark}[theorem]{Remark}
\newtheorem{notation}[theorem]{Notation}

\newtheorem{example}[theorem]{Example}

\newtheorem{definition}[theorem]{Definition} 

\numberwithin{equation}{section}

\newcommand{\FF}{\mathbb F}

\title{Some cases of Oort's conjecture about Newton polygons of curves}

\author{Rachel Pries}
\address{Department of Mathematics, 
Colorado State University, 
Fort Collins, CO 80523, USA, pries@colostate.edu}

\date{\today}

\begin{document}

\begin{abstract}
This paper contains a method to prove the existence of smooth curves in positive characteristic
whose Jacobians have unusual Newton polygons.  
Using this method, I give a new proof that 
there exist supersingular curves of genus $4$ in every prime characteristic.
More generally, the main result of the paper is that, 
for every $g \geq 4$ and prime $p$, 
every Newton polygon whose $p$-rank is at least $g-4$ occurs for a smooth curve of genus $g$ in characteristic $p$.
In addition, this method resolves some cases of Oort's conjecture about Newton polygons of curves.


\noindent
MSC20: primary 11G10, 11G20, 11M38, 14H10, 14H40; 
secondary 14G10, 14G15.
\end{abstract}



\maketitle

\section{Introduction}

Fix a prime number $p$.
The Newton polygon stratification of the moduli space $\mathcal{A}_g=\mathcal{A}_{g, \FF_p}$ of principally polarized abelian varieties of dimension $g$ in characteristic $p$ is well-understood, e.g., \cite{lioort}, \cite{JO00}, \cite{oort01}, 
\cite{oortNPgro}, and \cite{CO11}.
We say that a Newton polygon $\xi$ occurs on the moduli space $\mathcal{M}_g=\mathcal{M}_{g, \FF_p}$
of smooth curves of genus $g$ in characteristic $p$
if there exists a smooth curve of genus $g$ defined over 
$\overline{\FF}_p$ whose Jacobian has Newton polygon $\xi$.

There are some existence results showing that certain Newton polygons occur on ${\mathcal M}_g$;
see \cite{priesCurrent} for a survey with references.
Most of these results are
based on the computation of Newton polygons for curves with non-trivial automorphism group.
The supersingular Newton polygon is the one whose slopes are all $1/2$.
Recently, Harashita, Kudo, and Senda 
proved that there exists a supersingular smooth curve of genus $4$ in every prime characteristic \cite[Corollary~1.2]{khs20}.
For $g \geq 5$, there are almost no results about which Newton polygons occur on $\mathcal{M}_g$ 
for every prime characteristic.

This paper was inspired by a recent conversation with Oort, in which we discussed
a more geometric method for studying the Newton polygons that occur on $\mathcal{M}_g$. 
This method applies when the codimension of the Newton polygon stratum in $\mathcal{A}_g$
is small.
When the method applies, the proofs involved are short.

To illustrate this method, I give a new proof of \cite[Corollary~1.2]{khs20} here.
For $g \geq 1$, let $\sigma_g$ denote the supersingular Newton polygon of height $2g$.
Let $\mathcal{M}_g[\sigma_g]$ (resp.\ $\mathcal{A}_g[\sigma_g]$) denote the supersingular locus 
of $\mathcal{M}_g$ (resp.\ $\mathcal{A}_g$).

\begin{theorem} \label{T4supersingular}  
For every prime $p$: there exists a smooth curve of genus $4$ in characteristic $p$ that is supersingular; 
thus $\mathcal{M}_4[\sigma_4]$ is non-empty and its irreducible components have dimension at least $3$
in characteristic $p$. 
\end{theorem}

This method does not give a new proof of \cite[Theorem~1.1]{khs20}, which states that 
there exists a supersingular smooth curve of genus $4$ with $a$-number $a \geq 3$ in every positive characteristic $p > 3$.

\begin{proof}[Proof of Theorem~\ref{T4supersingular}]
Over $\overline{\mathbb{F}}_p$,
there exists a stable curve $C$ of genus $4$ that is singular and supersingular. 
For example, this can be produced by taking a chain of four supersingular elliptic curves, clutched together
at ordinary double points.
The Jacobian of $C$ is a principally polarized abelian variety of dimension 4 that is supersingular.
As such, it is represented by a point in ${\mathcal A}_4[\sigma_4] \cap T_4$, where
$T_4$ is the locus of Jacobians of stable curves of genus $4$. 
 
The codimension of $\mathcal{A}_4[\sigma_4]$ in $\mathcal{A}_4$ is $10-4=6$.
The codimension of $T_4 \cap \mathcal{A}_4$ in $\mathcal{A}_4$ is $10-9=1$.
Since $\mathcal{A}_4$ is a smooth stack, the codimension of a non-empty intersection of two substacks is at
most the sum of their codimensions \cite[page 614]{V:stack}.  It follows that
$\mathrm{codim}(\mathcal{A}_4[\sigma_4] \cap T_4, \mathcal{A}_4) \leq 7$.  
To summarize, $\mathcal{A}_4[\sigma_4] \cap T_4$ is non-empty and each of its irreducible components 
has dimension at least $3$.

Let $\delta$ denote the locus in $\mathcal{A}_4[\sigma_4] \cap T_4$
whose points represent the Jacobian of a curve $D$ that is stable but not smooth.
Since its Jacobian is an abelian variety, the curve $D$ has compact type, as defined in Section~\ref{Scompacttype}.
So its Jacobian is a principally polarized abelian fourfold that decomposes, 
with the product polarization.  

Then $\mathrm{dim}(\delta) \leq 2$.
This is because points in $\delta$ parametrize objects either of the form $E \oplus X$ where $E$ is a supersingular 
elliptic curve and $X$ is a supersingular abelian threefold, or of the form $X \oplus X'$ where $X,X'$ are supersingular 
abelian surfaces.  In the former case, the dimension is
$\mathrm{dim}(\mathcal{A}_1[\sigma_1] \oplus \mathcal{A}_3[\sigma_3]) = 0 + 2 =2$.
In the latter case, the dimension is $\mathrm{dim}(\mathcal{A}_2[\sigma_2] \oplus \mathcal{A}_2[\sigma_2]) =1 + 1=2$.
Since $2<3$, every generic geometric point of $\mathcal{A}_4[\sigma_4] \cap T_4$ represents 
the Jacobian of a supersingular curve of genus $4$ which is smooth.

Thus, in every prime characteristic, ${\mathcal M}_4[\sigma_4]$ is non-empty, and so
there exists a smooth curve of genus $4$ that is supersingular.
If $R$ is an irreducible component of 
${\mathcal M}_4[\sigma_4]$, then the image of $R$ under the Torelli morphism is open and dense 
in an irreducible component of $\mathcal{A}_4[\sigma_4] \cap T_4$; so $\mathrm{dim}(R) \geq 3$, which completes the proof.
\end{proof}

This proof is a special case of a method described in Section~\ref{Sgeneralmethod}; 
the full details are given above so that the proof is self-contained.
This method also applies for some Newton polygons that were not previously known to occur.
In particular, it applies to two new Newton polygons in dimension $4$, which yields the following result.

\begin{corollary} See Corollary~\ref{Cg4complete}.
Every symmetric Newton polygon in dimension $g=4$ occurs on $\mathcal{M}_4$,
in every prime characteristic $p$. 
\end{corollary}

\begin{remark}
The Newton polygon stratification of $\mathcal{M}_4$ is still not completely understood.
In particular, 
one would like to know whether all irreducible components of $\mathcal{M}_4[\sigma_4]$ have dimension $3$;
(apriori, $\mathcal{M}_4[\sigma_4]$ might have an irreducible component of dimension $4$).
In \cite[Corollary~4.4]{harashitarims}, Harashita proves that an irreducible 
component of $\mathcal{M}_4[\sigma_4]$ has dimension $3$ if
it contains a point representing a superspecial non-hyperelliptic curve.
\end{remark}

More generally, the method in Section~\ref{Sgeneralmethod} yields new applications
for curves of every genus $g \geq 4$, 
which are described briefly here.  See Section~\ref{Snotation} for more complete notation and background.
The \emph{$p$-rank} of a curve defined over an algebraically closed field of characteristic $p$
is the integer $f$ such that $p^f$ is the number of $p$-torsion points on its Jacobian; it equals the number
of slopes of $0$ in the Newton polygon.

The first application is a generalization of Corollary~\ref{Cg4complete}.
Realizing three new Newton polygons on $\mathcal{M}_g$, for each $g > 4$, 
yields the following result.

\begin{theorem} See Theorem~\ref{Tprank34}.  For $g \geq 4$, 
every symmetric Newton polygon in dimension $g$ with $p$-rank $f \geq g-4$ occurs on $\mathcal{M}_g$,
in every prime characteristic $p$.
\end{theorem}

The second application of the method in Section~\ref{Sgeneralmethod} is about Oort's conjecture for Newton polygons of curves.
For $i=1,2$, let $g_i$ be a positive integer, and let $\xi_i$ be a symmetric Newton polygon of height $2g_i$.
Let $g=g_1+g_2$, and 
let $\xi=\xi_1 \oplus \xi_2$ denote the symmetric Newton polygon of height $2g$ obtained 
by taking the union of the slopes of $\xi_1$ and $\xi_2$, 
letting the multiplicity of the slope $\lambda$ in $\xi$ be the sum of its multiplicities in $\xi_1$ and $\xi_2$.

\begin{conjecture}  \cite[Conjecture~8.5.7]{oort05}
Oort's conjecture for Newton polygons of curves states that:
If $\xi_i$ occurs on ${\mathcal M}_{g_i}$ for $i=1,2$, then
$\xi_1 \oplus \xi_2$ occurs on ${\mathcal M}_{g_1+g_2}$.
\end{conjecture}

In Proposition~\ref{Pslopes1overd}, for every $g \geq 4$, in every prime characteristic, 
I prove Oort's conjecture is true when $\xi_1$ and $\xi_2$ are chosen as follows:
let $d=g-1$ and let $\xi_1$ be the Newton polygon with slopes $1/d$ and $(d-1)/d$, each with multiplicity $d$; 
and let $\xi_2 = \sigma_1$ be the Newton polygon with two slopes of $1/2$.
This case of Oort's conjecture is more significant than earlier work from a geometric perspective, 
because the Newton polygon stratum for $\xi=\xi_1\oplus \xi_2$ has codimension $g+1$ in ${\mathcal A}_g$ in this case,
Remark~\ref{Rcodimension}.

Both Theorem~\ref{Tprank34} and Corollary~\ref{Cg4complete} rely on Proposition~\ref{Pslopes1overd}.
Also, using Proposition~\ref{Pslopes1overd}, it is
possible to realize new Newton polygons for curves of genus four to seven, Corollary~\ref{Capplication}.

In particular, when $g=5$, these applications show that four Newton polygons occur on $\mathcal{M}_5$, 
in every prime characteristic $p$, which were not previously known to occur.  In Example~\ref{Eg5prank0},
I include some examples when $g=5$ and the $p$-rank is $0$, to show some situations where 
the method in Section~\ref{Sgeneralmethod} does not apply.

\paragraph{Acknowledgment}
Pries was partially supported by NSF grant DMS-22-00418 and would like to thank Oort for the conversation that inspired this paper and thank Kedlaya and the anonymous referee for helpful suggestions.

\section{Notation and background} \label{Snotation}

Fix a prime number $p$.  Let $g$ be a positive integer.
Suppose $A$ is an abelian variety of dimension $g$ defined over an algebraically closed field $k$ of characteristic $p$.

\subsection{The $p$-rank and Newton polygon} \label{Snewtonpoly}

The \emph{$p$-rank} of $A$ is the integer $f$ such that $\#A[p](k)=p^f$.   
More generally, if $A$ is a semi-abelian variety, then its $p$-rank is 
$f_A = \mathrm{dim}_{\mathbb{F}_p}(\mathrm{Hom}(\mu_p, A))$
where $\mu_p$ is the kernel of Frobenius on the multiplicative group $\mathbb{G}_m$.

For each pair of non-negative relatively prime integers $c$ and $d$, 
fix a $p$-divisible group $G_{c,d}$ of codimension $c$, dimension $d$, and thus height $c + d$. 
The slope of $G_{c,d}$ is $\lambda = d/(c+d)$.

Suppose $A$ is a principally polarized abelian variety of dimension $g$.
By the classification result of Dieudonn\'e-Manin \cite{man63}, there is an isogeny of $p$-divisible groups
$A[p^\infty] \sim \oplus_{\lambda=d/(c+d)} G_{c,d}^{m_\lambda}$.
The \emph{Newton polygon} of $A$ is the data of the set $\{m_\lambda\}$.  
The Newton polygon of a curve is that of its Jacobian.

The Newton polygon $\xi$ of $A$ is drawn as a lower convex polygon, with endpoints $(0,0)$ and $(2g,g)$, 
whose slopes are the values of $\lambda$ occurring with multiplicity $(c+d)m_\lambda$.
It has integer breakpoints and is symmetric, meaning that $m_\lambda = m_{1-\lambda}$ for every slope $\lambda$.  
A \emph{Newton polygon in dimension $g$} means a symmetric Newton polygon that satisfies these conditions.

The abelian variety A is \emph{supersingular} if and only if $\lambda = 1/2$ is the only slope of its Newton polygon $\xi$.
The $p$-rank of $A$ is the multiplicity of the slope $0$ in $\xi$.

\begin{notation} \label{Nnud0}
Let $ord$ denote the ordinary Newton polygon in dimension $1$; it has slopes $0$ and $1$
and $p$-divisible group $G_{0,1} \oplus G_{1,0}$.
Let $ss$ denote the supersingular Newton polygon in dimension $1$; it has 
slopes $1/2$ and $p$-divisible group $G_{1,1}$. 

For $g \geq 2$, let $\sigma_g=ss^g$ denote the supersingular Newton polygon in dimension $g$; 
it has slope $1/2$ with multiplicity $2g$, and has $p$-divisible group $G_{1,1}^g$.
For $d \geq 3$, let $\nu_{d}^0$ be the Newton polygon with slopes $1/d$ and $(d-1)/d$, each with multiplicity 
$d$; it has $p$-divisible group $G_{1,d-1} \oplus G_{d-1,1}$.
\end{notation}

\begin{definition}
Let $\xi$ be a Newton polygon in dimension $g$.
A \emph{partition} of $\xi$ consists of 
a partition $g=g_1+g_2$ of $g$ into two positive integers, together with Newton polygons
$\xi_1, \xi_2$ in dimension $g_1, g_2$ such that $\xi = \xi_1 \oplus \xi_2$.
\end{definition}

\subsection{Curves of compact type} \label{Scompacttype}

Suppose $C$ is a stable curve of arithmetic genus $g$.
We say $C$ has \emph{compact type} if its dual graph is a tree.
By \cite[Section~9.2, Example~8]{BLR}, the Jacobian of $C$ is an abelian variety if and only if 
$C$ has compact type; if not, the Jacobian of $C$ is a semi-abelian variety and its toric rank is positive, 
which implies that its $p$-rank is positive.

\subsection{Moduli spaces}

Let ${\mathcal A}_g = \mathcal{A}_{g, \FF_p}$ be the moduli space
of principally polarized abelian varieties of dimension $g$ in characteristic $p$.  It has dimension $g(g+1)/2$.
Recall that ${\mathcal A}_g$ is a smooth stack.

Let $\mathcal{M}_g$ (resp.\ $\mathcal{M}_g^{ct}$, resp.\ $\overline{{\mathcal M}}_g$)
denote the moduli space of projective connected curves of genus $g$ that are smooth (resp.\ of compact type, resp.\ stable).
The dimension of $\mathcal{M}_g$ is $3g-3$ if $g \geq 2$.
Let $\tau: \mathcal{M}_g^{ct} \to {\mathcal{A}}_g$ be the Torelli morphism, which takes a curve 
to its Jacobian.
The \emph{closed Torelli locus} $T_g$ is the image of ${\mathcal{M}}_g^{ct}$ under $\tau$. 
The \emph{open Torelli locus} $T_g^\circ$ is the image of ${\mathcal{M}}_g$ under $\tau$.    

Note that the fibers of $\tau$ over $T_g - T_g^\circ$ are not finite. 
When building a singular curve of compact type from curves of smaller genus, it is necessary 
to choose points at which to clutch them together.  
This choice of points does not affect the Jacobian of such a curve.
The key idea in the proof of Proposition~\ref{Pnunonempty} is to work with intersections in $\mathcal{A}_g$
rather than $\overline{\mathcal{M}}_g$.

Let $\mathcal{A}_g^f$ (resp.\ $\mathcal{M}_g^f$) denote the $p$-rank $f$ locus of $\mathcal{A}_g$ 
(resp.\ $\mathcal{M}_g$).
In most cases, it is not known whether $\mathcal{M}_g^f$ is irreducible.
By \cite[Theorem~2.3]{FVdG}, the dimension of each irreducible component
of $\mathcal{M}_g^f$ is $2g-3 +f$.

\subsection{Newton polygon stratification} 

Let $\mathcal{M}_g[\xi]$ (resp.\ $\mathcal{A}_g[\xi]$) denote the stratum of $\mathcal{M}_g$ (resp.\ $\mathcal{A}_g$) 
whose points represent objects with Newton polygon $\xi$.
By the purity result of de Jong and Oort \cite[Theorem~4.1]{JO00}, if the Newton polygon changes on a family of 
abelian varieties, then it changes already in codimension $1$.

By \cite[Theorem~4.1]{oort01}, the codimension of $\mathcal{A}_g[\xi]$ in $\mathcal{A}_g$ 
equals the number of lattice points below $\xi$.
In particular, the supersingular locus $\mathcal{A}_g[\sigma_g]$ has dimension $\lfloor g^2/4 \rfloor$ 
\cite[Section 4.9]{lioort}.

\begin{definition}
If $\xi$ is a Newton polygon in dimension $g$, the \emph{$e$-dimension} of 
$\mathcal{M}_g[\xi]$ is:
\[e(\xi, \mathcal{M}_g) := \mathrm{max}\{0, 3g-3 - \mathrm{codim}(\mathcal{A}_g[\xi], \mathcal{A}_g)\},\]
unless $g=1$ and $\xi=ord$, in which case $e(ord, \mathcal{M}_{1,1})=1$.
\end{definition}

Here is the intuition behind this definition.  
If $T_g^{\circ}$ is dimensionally transverse to $\mathcal{A}_g[\xi]$, then $e(\xi, \mathcal{M}_g)$ 
will be the dimension of each irreducible component of $\mathcal{M}_g[\xi]$.
A priori, it is possible that $\mathcal{M}_g[\xi]$ is empty, or that an irreducible component of $\mathcal{M}_g[\xi]$ has dimension
greater than $e(\xi, \mathcal{M}_g)$, if $T_g^{\circ}$ is not dimensionally transverse to $\mathcal{A}_g[\xi]$.

\section{Intersection in the moduli space of abelian varieties} \label{Sgeneralmethod}

\subsection{A method for realizing Newton polygons}

The following method for realizing Newton polygons does not appear in the literature.

\begin{proposition} \label{Pnunonempty}
Suppose $\xi$ is a Newton polygon in dimension $g$.
Suppose that: \\
(a) there is a partition $g=g_1 + g_2$ and $\xi=\xi_1 \oplus \xi_2$ such that 
$\mathcal{M}^{ct}_{g_i}[\xi_i]$ is non-empty for $i=1,2$; \\ 
(b) for every partition $\xi=\xi_1 \oplus \xi_2$: either (i) $\mathcal{M}^{ct}_{g_i}[\xi_i]$ is empty for $i=1$ or $i=2$; or\\ 
(ii) each of the irreducible components of $\mathcal{M}^{ct}_{g_i}[\xi_i]$ has dimension $e(\xi_i, \mathcal{M}_{g_i})$ 
for $i=1,2$; and 
\begin{equation} \label{Eformula}
e(\xi_1, \mathcal{M}_{g_1})+ e(\xi_2, \mathcal{M}_{g_2}) < e(\xi, \mathcal{M}_g).
\end{equation}
Then $\xi$ occurs on $\mathcal{M}_g$.
\end{proposition}

\begin{proof}
By hypothesis (a), $\mathcal{A}_g[\xi] \cap T_g$ is non-empty, because it contains a point representing
$\mathrm{Jac}(C_1) \oplus \mathrm{Jac}(C_2)$, where $C_i$ is a compact-type curve of genus $g_i$ 
having Newton polygon $\xi_i$, for $i=1,2$.

Let $S$ be an irreducible component of $\mathcal{A}_g[\xi] \cap T_g$.
Since $\mathcal{A}_g$ is a smooth stack, the codimension of a non-empty intersection of two substacks is at
most the sum of their codimensions  \cite[page 614]{V:stack}. 
This implies that $\mathrm{codim}(S, T_g) \leq \mathrm{codim}(\mathcal{A}_g[\xi], \mathcal{A}_g)$.
This is equivalent to $\mathrm{dim}(S) \geq e(\xi, \mathcal{M}_g)$.

Let $\delta$ denote the locus in $S$
whose points represent the Jacobian of a stable curve $D$ that is not smooth.
Since $S \subset \mathcal{A}_g$, the curve $D$ has compact type.
By \cite[Section~9.2, Example~8]{BLR}, the Jacobian $J$ of $D$ is of the form 
$X_1 \oplus X_2$, where $X_i$ is the Jacobian of a curve of genus 
$g_i$ for some partition $g=g_1+g_2$.  Also the Newton polygon of $J$ is $\xi_1 \oplus \xi_2$ 
for some partition $\xi=\xi_1\oplus \xi_2$, with $\xi_i$ the Newton polygon of $X_i$ for $i=1,2$.

Thus $\delta$ is the union, for all partitions $g=g_1+g_2$ and $\xi = \xi_1 \oplus \xi_2$,
of the image, under the Torelli morphism $\tau$, of $\mathcal{M}^{ct}_{g_1}[\xi_1] \times \mathcal{M}^{ct}_{g_2}[\xi_2]$.
For each partition, by hypothesis (b), $\mathcal{M}^{ct}_{g_1}[\xi_1] \times \mathcal{M}^{ct}_{g_2}[\xi_2]$ is 
either empty, or it has dimension $e(\xi_1, \mathcal{M}_{g_1})+e(\xi_2, \mathcal{M}_{g_2})$, which
is strictly less than $e(\xi, \mathcal{M}_g)$ by hypothesis.
Since $\tau$ is an embedding, $\mathrm{dim}(\delta) < e(\xi, \mathcal{M}_g)$.
So the generic geometric point of $S$ represents 
the Jacobian of a curve of genus $g$ that has Newton polygon $\xi$ and that is smooth.
\end{proof}

Theorem~\ref{T4supersingular} provides one illustration of Proposition~\ref{Pnunonempty}.
One can also use Proposition~\ref{Pnunonempty} to quickly verify that there exists a smooth supersingular curve of genus $2$
(resp.\ genus $3$).

It is sometimes convenient to re-express the condition in \eqref{Eformula} in terms of the codimensions of the strata.
Suppose $g=g_0 \geq 3$ and $g_0=g_1+g_2$, $\xi_0=\xi_1 \oplus \xi_2$ is a partition.
For $i=0,1,2$, let $c_i = \mathrm{codim}(\mathcal{A}_{g_i}[\xi], \mathcal{A}_{g_i})$.

\begin{lemma} \label{Lcodim}
Suppose $g_0 \geq 3$.
\begin{enumerate}
\item If $\xi_1 = ord^{g_1}$, then \eqref{Eformula} is satisfied.
\item Suppose $g_1=1$ and $\xi_1 = ss$.  Suppose $c_i \leq 3g_i-3$ for $i=0,2$. 
Then \eqref{Eformula} is equivalent to $c_0 < c_2 + 3$.
\item Suppose $c_i \leq 3g_i-3$ for $i=0,1,2$.
Then \eqref{Eformula} is equivalent to $c_0 < c_1 + c_2 + 3$.
\end{enumerate}
\end{lemma}

\begin{proof} 
\begin{enumerate}
\item In this case, $e(\xi_1, \mathcal{M}_{g_1}) = 3g_1 -3$ if $g_1 \geq 2$ and 
$e(\xi_1, \mathcal{M}_{1,1}) = 1$ if $g_1 =1$.
Since $\xi_1$ is ordinary, $c_2 = c_0$. 
So \eqref{Eformula} is equivalent to 
$e(\xi_1, \mathcal{M}_{g_1}) + (3g_2-3) < 3g_0-3$, which is true because $g_1+g_2 = g_0$.
\item This is true by direct computation because $e(ss, \mathcal{M}_{1,1}) = 0$ and $g_0=g_2+1$.
\item Direct computation. 
\end{enumerate}
\end{proof}

Proposition~\ref{Pnunonempty}, together with Lemma~\ref{Lcodim}(1), may be helpful in the future to reduce 
questions about the Newton polygon stratification of $\mathcal{M}_g$ to the case of $p$-rank $0$.

\subsection{Some cases of Oort's conjecture}

Proposition~\ref{Pnunonempty} has some applications for Oort's conjecture for Newton polygons.
Recall from Notation~\ref{Nnud0} that $\nu_d^0$ is the Newton polygon in dimension $d$ with slopes $1/d$ and $(d-1)/d$; 
its $p$-divisible group is $G_{1,d-1} \oplus G_{d-1,1}$.

\begin{proposition} \label{Pslopes1overd}
Let $d \geq 3$.  For every prime characteristic $p$, 
Oort's conjecture is true for $\xi_1 = \nu_d^0$ and $\xi_2 = ss$;
thus, if there exists a smooth curve of genus $d$ with Newton polygon $\nu_d^0$,
then there exists a smooth curve of genus $g=d+1$ having Newton polygon 
$\nu_d^0 \oplus ss$.
\end{proposition}

\begin{proof}
The Newton polygon $\xi_1=\nu_d^0$ has $p$-rank $0$.  
It is the most generic Newton polygon of $p$-rank $0$ in $\mathcal{A}_d$, and 
the $p$-rank $0$ locus has codimension $d$ in $\mathcal{A}_d$.
Thus $e(\xi_1, \mathcal{M}_d) = 2d-3$.

By hypothesis, $\mathcal{M}_d[\xi_1]$ is non-empty.
Each irreducible component of $\mathcal{M}_d[\xi_1]$ is open and dense in an irreducible component 
of the $p$-rank $0$ locus $\mathcal{M}_d^0$, which has dimension $2d-3$ by \cite[Theorem~2.3]{FVdG}.

Let $g=d+1$ and consider the Newton polygon $\xi= \xi_1 \oplus ss$.
Let $\xi_2 = ss$.  The only partition of $\xi$ is $\xi_1 \oplus \xi_2$, because $\xi_1$ is 
indecomposable as a symmetric Newton polygon.
Hypothesis (a) is satisfied for this partition because $\mathcal{M}_d[\xi_1]$ is non-empty by 
hypothesis and because there exists a supersingular elliptic curve.

Note that $\mathcal{A}_g[\xi]$ has codimension $1$ in the $p$-rank $0$ locus of $\mathcal{A}_g$.
So $e(\xi, \mathcal{M}_g) = 2g-4 = 2d-2$.
Then $e(\xi_1, \mathcal{M}_d) + e(\xi_2, \mathcal{M}_{1,1}) = 2d-3 + 0$, which is smaller than
$e(\xi, \mathcal{M}_g)$.  Thus hypothesis (b) of Proposition~\ref{Pnunonempty} is also satisfied and the result follows.
\end{proof}

Proposition~\ref{Pslopes1overd} is used in the next section
to prove the existence of smooth curves of genus $g$
having a Newton polygon whose stratum in $\mathcal{A}_g$ has codimension larger than $g$. 

\begin{remark} \label{Rcodimension}
By \cite[Theorem~6.4]{priesCurrent}, Oort's conjecture is true for $\xi_1 = \nu_d^0$ and $\xi_2=ord^e$ for any $e \geq 1$
and every prime characteristic $p$.
Here is the reason why Proposition~\ref{Pslopes1overd} is more significant than that result from a geometric perspective.
On $\mathcal{A}_d$, the Newton polygon stratum for $\xi_1=\nu_d^0$ has codimension $d$.
The Newton polygon $\xi = \xi_1 \oplus \xi_2 = \nu_d^0 \oplus ord^e$ is the most generic one with $p$-rank $e$; 
thus $\mathrm{codim}(\mathcal{A}_{d+e}[\xi], \mathcal{A}_{d+e}) = d$.
In other words, $\mathrm{codim}(\mathcal{A}_{d+e}[\xi], \mathcal{A}_{d+e}) =  
\mathrm{codim}(\mathcal{A}_{d}[\xi_1], \mathcal{A}_{d})$ 
in \cite[Theorem~6.4]{priesCurrent}.

In contrast, in Proposition~\ref{Pslopes1overd}, $\xi_2 = ss$.  
The Newton polygon $\xi=\nu_d^0 \oplus ss$ is the second most generic one with $p$-rank $0$;
thus $\mathrm{codim}(\mathcal{A}_{d+1}[\xi], \mathcal{A}_{d+1}) = d+2$, which is 
greater by two than $\mathrm{codim}(\mathcal{A}_d[\xi_1], \mathcal{A}_d)$.
When a Newton polygon stratum has larger codimension, it is more difficult to show that it intersects the open Torelli locus.
\end{remark}

\section{Applications} \label{Sapplication}

The goal of this section is to demonstrate some new Newton polygons that occur for Jacobians of smooth curves.

\subsection{Some new Newton polygons of codimension $g+1$}

\begin{corollary} \label{Capplication} In every prime characteristic $p$: 
\begin{enumerate}
\item There exists a smooth curve of genus $4$ with Newton polygon 
$\nu_3^0 \oplus ss$.
\item There exists a smooth curve of genus $5$ with Newton polygon 
$\nu_4^0 \oplus ss$.
\item If $p \equiv 3,4,5,9 \bmod 11$, then
there exists a smooth curve of genus $6$ with Newton polygon $\nu_5^0 \oplus ss$.
\item If $p \equiv 2,4 \bmod 7$, then
there exists a smooth curve of genus $7$ with Newton polygon $\nu_6^0 \oplus ss$.
\end{enumerate}
\end{corollary}

\begin{proof}
This is immediate from Proposition~\ref{Pslopes1overd} once the existence of a smooth curve 
of genus $d=g-1$ with Newton polygon $\nu_d^0$ (slopes $1/d$ and $(d-1)/d$) is verified.  
For $d=3$, this is true because $T_3^\circ$ is open and dense in $\mathcal{A}_3$ and $\nu_3^0$ 
is indecomposable as a symmetric Newton polygon.
For $d=4$, this is verified in \cite[Lemma~5.3]{AP:gen}.
For $d=5$, when $p \equiv 3,4,5,9 \bmod 11$, this is verified in \cite[Theorem~1.2]{LMPT1}.
For $d=6$, when $p \equiv 2,4 \bmod 7$, this is verified in 
\cite[Theorem~7.4]{LMPT2}.
\end{proof}

\subsection{Some new Newton polygons for arbitrary $g$}

Suppose $g \geq 4$.
In this section, I realize three new Newton polygons on $\mathcal{M}_g$ for each $g \geq 5$ and two new Newton polygons 
on $\mathcal{M}_4$, for every prime characteristic $p$: 
Some consequences of this are: 

Theorem~\ref{Tprank34}: every Newton polygon having $p$-rank $f \geq g-4$ occurs on $\mathcal{M}_g$. 

Corollary~\ref{Cg4complete}:
every Newton polygon in dimension $g=4$ occurs on $\mathcal{M}_4$.

Corollary~\ref{Cg5completeprankpos}:
every Newton polygon in dimension $g=5$ with $p$-rank $f >0$ occurs on $\mathcal{M}_5$.

Suppose $\xi$ is a Newton polygon in dimension $g$.
The following lemma is useful when the Newton polygon stratum for $\xi$ has small codimension in $\mathcal{A}_g$.

\begin{lemma} \label{Lrightdim}
Let $c=\mathrm{codim}(\mathcal{A}_g[\xi], \mathcal{A}_g)$.
If $c \leq 3$, then $\xi$ occurs on $\mathcal{M}_g$ and every irreducible component of $\mathcal{M}_g[\xi]$ has 
codimension $c$ in $\mathcal{M}_g$.
If $c = 4$, then the same is true for $\xi=ord^{g-4} \oplus \nu_4^0$.
\end{lemma}

\begin{proof}
For $c=0$ (resp.\ $c=1$, resp.\ $c=2$), 
the only option for the Newton polygon $\xi$ is $ord^g$ (resp.\ $ord^{g-1} \oplus ss$, resp.\ $ord^{g-2} \oplus ss^2$);
and this is the only Newton polygon 
in dimension $g$ having $p$-rank $g$ (resp.\ $g-1$, resp.\ $g-2$).
So in these cases, $\xi$ occurs on $\mathcal{M}_g$ and every irreducible component
of $\mathcal{M}_g[\xi]$ has 
codimension $c$ in $\mathcal{M}_g$ by \cite[Theorem~2.3]{FVdG}.

For $c=3$, the only option for the Newton polygon $\xi$ is $ord^{g-3} \oplus \nu_3^0$.
By \cite[Corollary~5.5]{AP:gen}, this Newton polygon $\xi$ occurs on $\mathcal{M}_g$ and
every irreducible component of $\mathcal{M}_g[\xi]$ is open and dense in an irreducible component of $\mathcal{M}_g^{g-3}$.
By \cite[Theorem~2.3]{FVdG}, every irreducible component of $\mathcal{M}_g^{g-3}$ has codimension 
$3$ in $\mathcal{M}_g$. 
Furthermore, $\xi$ is the generic Newton polygon on every irreducible component of $\mathcal{M}_g^{g-3}$ by \cite[Corollary~5.5]{AP:gen}.

If $c=4$ and $\xi=ord^{g-4} \oplus \nu_4^0$, then by \cite[Corollary~6.5]{priesCurrent}, $\xi$ occurs on 
$\mathcal{M}_g$ and every irreducible component of $\mathcal{M}_g[\xi]$ is open and dense in an irreducible component of $\mathcal{M}_g^{g-4}$.
By \cite[Theorem~2.3]{FVdG}, every irreducible component of $\mathcal{M}_g^{g-4}$ 
has codimension $4$ in $\mathcal{M}_g$.
\end{proof}

Here is the main result of the paper.

\begin{theorem} \label{Tprank34}
Let $g \geq 4$.  
These symmetric Newton polygons occur on $\mathcal{M}_g$, in every prime characteristic $p$:
\begin{enumerate}
\item $\xi=ord^{g-3} \oplus ss^3$;
\item $\xi=ord^{g-4} \oplus \nu_3^0 \oplus ss$; and
\item $\xi= ord^{g-4} \oplus ss^4$.
\end{enumerate}
As a result, every Newton polygon with $p$-rank $f \geq g-4$ occurs on $\mathcal{M}_g$.
\end{theorem}

\begin{proof}
The Newton polygons in parts (1)-(3) are the only ones with $p$-rank $f \geq g-4$ that are not covered in
Lemma~\ref{Lrightdim}.  So the final claim follows from the occurrence of these Newton polygons on $\mathcal{M}_g$.

Throughout this proof, Lemma~\ref{Lrightdim} will be used without comment to check the first part of hypothesis~(b)(ii).
Also, the hypothesis $c_i \leq 3g_i-3$ used in Lemma~\ref{Lcodim} 
is true in all cases below, except for the case $g_i=1$ and $\xi_i =ord$, which is handled separately.
The length of the proof is due to the number of partitions of $\xi$.
Write $g=g_0$, $\xi=\xi_0$, and $c=c_0$.

\begin{enumerate}
\item Let $\xi=ord^{g-3} \oplus ss^3$.  If $\xi$ occurs on $\mathcal{M}_g$, 
then its stratum must have codimension $1$ in $\mathcal{M}_g^{g-3}$ by purity and \cite[Corollary~5.5]{AP:gen}.
So if $\xi$ occurs on $\mathcal{M}_g$, then $\mathcal{M}_g[\xi]$ has codimension $c=4$ in $\mathcal{M}_g$.

The result now follows by induction on $g$.
For the base case $g=3$, the Newton polygon $ss^3$ occurs on $\mathcal{M}_3$; 
see e.g., \cite[Theorem~5.12(2)]{oorthypsup}.
Suppose now that $g \geq 4$ and that the result is true for all $3 \leq g' < g$.
The inductive step relies on Proposition~\ref{Pnunonempty}.
Hypothesis (a) is true for the partition $\xi_1=ord^{g-3}$ and $\xi_2 = ss^3$.   

For $1 \leq t \leq g-3$, consider the partition $\xi_1 = ord^t$ and $\xi_2 = ord^{g-3-t} \oplus ss^3$.
Then hypothesis (b) is satisfied by Lemma~\ref{Lcodim}(1) and the inductive hypothesis.

Consider the partition $\xi_1=ss$ and $\xi_2=ord^{g-3} \oplus ss^2$. 
Hypothesis (b) is true by Lemma~\ref{Lrightdim} and Lemma~\ref{Lcodim}(2), because $4=c < c_2 + 3 = 5$.

By symmetry, the only remaining partitions are $\xi_1 = ord^t \oplus ss$ and $\xi_2 = ord^{g-3-t} \oplus ss^2$, 
for some $1 \leq t \leq g-3$.  Then $c_1=1$ and $c_2=2$.
Hypothesis (b) is true by Lemma~\ref{Lrightdim} and Lemma~\ref{Lcodim}(3), because $4=c< c_1+ c_2 + 3 = 6$.

\item Let $\xi=ord^{g-4} \oplus \nu_3^0 \oplus ss$.  Then $c=5$.  
If $\xi$ occurs on $\mathcal{M}_g$, then its stratum is contained in $\mathcal{M}_g^{g-4}$, 
which has codimension $4$ in $\mathcal{M}_g$.
However, it is not known whether $ord^{g-4} \oplus \nu_4^0$ is 
the generic Newton polygon on every irreducible component of $\mathcal{M}_g^{g-4}$.
For this reason, if $\xi$ occurs on $\mathcal{M}_g$, then 
each irreducible component of $\mathcal{M}_g[\xi]$ has codimension either $5$ (dimensionally transverse intersection) 
or $4$ in $\mathcal{M}_g$.

The proof will be by induction on $g$.
For the base case $g=4$, then $\xi=\nu_3^0 \oplus ss$ which occurs on $\mathcal{M}_4$ by Corollary~\ref{Capplication}(1).
Suppose that $g \geq 5$ and the result is true for all $4 \leq g' < g$.

The proof follows the strategy of Proposition~\ref{Pnunonempty} with a few adjustments.
Hypothesis (a) is true for the partition $\xi_1 = ord^{g-4}$ and $\xi_2 = \nu_3^0 \oplus ss$.

For hypothesis (b), consider the partition $\xi_1 = ord^t$ and $\xi_2 = ord^{g-4-t} \oplus \nu_3^0 \oplus ss$ 
for $1 \leq t \leq g-4$.  Then $c_1=0$ and $c_2=5$.  The inequality 
$5 = c < c_1 + c_2 + 3=8$ has a margin of error that is large enough to accommodate the possibility that 
$\mathcal{M}_{g-t}[\xi_2]$ has codimension $4$ rather than $5$;
this is true even when $t=1$, in which case the discrepancy between $3g_1-3$ and $1$ 
reduces the value on the right hand side of the inequality by $1$.
Thus the result follows by the inductive hypothesis and similar arguments as in Proposition~\ref{Pnunonempty}.

Next, consider the partition $\xi_1=ss$ and $\xi_2 = ord^{g-4} \oplus \nu_3^0$.
Then $c_2 = 3$.  Since $5=c< c_2+3=6$, hypothesis (b) is satisfied by Lemmas~\ref{Lcodim}(2) and \ref{Lrightdim}.

By symmetry, the only remaining partitions are $\xi_1= ord^t \oplus ss$ and $\xi_2=ord^{g-4-t} \oplus \nu_3^0$,
for $1 \leq t \leq g-4$.
Then $c_1=1$ and $c_2=3$.  Since $5=c < c_1 +c_2 + 3=7$, 
hypothesis (b) is satisfied by Lemmas~\ref{Lcodim}(3) and \ref{Lrightdim}.

\item Let $\xi=ord^{g-4} \oplus ss^4$.  Then $c=6$.
If $\xi$ occurs on $\mathcal{M}_g$, then its stratum is contained in $\mathcal{M}_g^{g-4}$, 
which has codimension $4$ in $\mathcal{M}_g$.
By \cite[Theorem~4.2(b) and Lemma~5.2]{AP:gen}, for each irreducible component of $\mathcal{M}_g^{g-4}$, 
the generic Newton polygon is either $ord^{g-4} \oplus \nu_4^0$ or $ord^{g-4} \oplus \nu_3^0 \oplus ss$.
By purity, if $\xi$ occurs on $\mathcal{M}_g$, this implies that 
each irreducible component of $\mathcal{M}_g[\xi]$ has codimension either $6$ (dimensionally transverse intersection) 
or $5$ in $\mathcal{M}_g$.

The proof will be by induction on $g$.
The base case $g=4$ is \cite[Corollary~1.2]{khs20}.
Suppose that $g \geq 5$ and the result is true for all $4 \leq g' < g$.
Hypothesis (a) is true for the partition $\xi_1 = ord^{g-4}$ and $\xi_2 = ss^4$.

For hypothesis (b), consider the partition $\xi_1 = ord^t$ and $\xi_2 = ord^{g-4-t} \oplus ss^4$ 
for $1 \leq t \leq g-4$.  Then $c_1=0$ and $c_2=6$.  The inequality 
$6 = c < c_1 + c_2 + 3 =9$ has a margin of error that is large enough to accommodate the possibility that 
$\mathcal{M}_{g-t}[\xi_2]$ has codimension $5$ rather than $6$;
this is true even when $t=1$, in which case the discrepancy between $3g_1-3$ and $1$ 
reduces the value on the right hand side of the inequality by $1$.
Thus the result follows by the inductive hypothesis and similar arguments as in Proposition~\ref{Pnunonempty}.

Next, consider the partition $\xi_1=ss$ and $\xi_2 = ord^{g-4} \oplus ss^3$.
Then $c_2 = 4$.  Since $6=c< c_2+3=7$, hypothesis (b) is satisfied by Lemmas~\ref{Lcodim}(2) and \ref{Lrightdim}.

Next, consider the partition $\xi_1= ord^t \oplus ss$ and $\xi_2=ord^{g-4-t} \oplus ss^3$,
for $1 \leq t \leq g-4$.
Then $c_1=1$ and $c_2=4$.
Since $6=c < c_1 + c_2 + 3=8$, 
hypothesis (b) is satisfied by Lemmas~\ref{Lcodim}(3) and \ref{Lrightdim}.

The only remaining type of partition is 
$\xi_1= ord^t \oplus ss^2$ and $\xi_2=ord^{g-4-t} \oplus ss^2$, for $0 \leq t \leq g-4$. 
Then $c_1=c_2 = 2$.
Since $6=c < c_1 +c_2 + 3=7$, 
hypothesis (b) is satisfied by Lemmas~\ref{Lcodim}(3) and \ref{Lrightdim}.
\end{enumerate}
This completes the proof.
\end{proof}

\begin{corollary} \label{Cg4complete}
Every symmetric Newton polygon in dimension $g=4$ occurs on $\mathcal{M}_4$, 
in every prime characteristic $p$. 
\end{corollary}

\begin{proof}
This is immediate from the final claim of Theorem~\ref{Tprank34} because the condition $f \geq g-4$ is vacuous when $g=4$.
\end{proof}

\subsection{The case of genus $5$}

\begin{corollary} \label{Cg5completeprankpos}
Every symmetric Newton polygon in dimension $g=5$ whose $p$-rank is positive occurs on $\mathcal{M}_5$,
in every prime characteristic $p$. 
\end{corollary}

\begin{proof}
This is immediate from the final claim of Theorem~\ref{Tprank34} because the condition $f \geq g-4$ is true by hypothesis
when $g=5$ and $f >0$.
\end{proof}

\begin{example} \label{Eg5prank0}
Here are some examples in the situation of genus $5$ and $p$-rank $0$.
\begin{enumerate}
\item When $\xi$ has slopes $(1/5, 4/5)$ with $p$-divisible group $G_{1,4} \oplus G_{4,1}$:
the method of Proposition~\ref{Pnunonempty} does not apply because $\xi$ is indecomposable as a symmetric Newton 
polygon and so hypothesis (a) is not satisfied.
For $p \equiv 3,4,5,9 \bmod 11$, it is known that this Newton polygon occurs on $\mathcal{M}_5$ by
\cite[Theorem~5.4]{LMPT1}.

\item When $\xi = \nu_4^0 \oplus ss$: then $\xi$ occurs on $\mathcal{M}_5$ by Corollary~\ref{Capplication}(2).

\item When $\xi$ has slopes $1/3, 1/2, 2/3$ with  $p$-divisible group $G_{1,2} \oplus G_{2,1} \oplus G_{1,1}^2$:
at this time, it is not known whether or not hypothesis (b) of Proposition~\ref{Pnunonempty} is satisfied. 
In more detail, consider the partition where $\xi_1$ has $p$-divisible group 
$G_{1,2} \oplus G_{2,1} \oplus G_{1,1}$ and $\xi_2=ss$. 
Note that $e(\xi_1, \mathcal{M}_4) =4$, $e(\xi_2, \mathcal{M}_{1,1})=0$, and $e(\xi, \mathcal{M}_5)=5$ and 
$0+4 < 5$.
However, it is not currently known whether hypothesis (b) is true in this case;
a priori, it is possible that an irreducible component of $\mathcal{M}_4[\xi_1]$ 
is open and dense in an irreducible component of $\mathcal{M}_4^0$, which has dimension $5$.

\item When $\xi$ has slopes $(2/5, 3/5)$ with $p$-divisible group $G_{2,3} \oplus G_{3,2}$:
the method of Proposition~\ref{Pnunonempty} does not apply because $\xi$ is indecomposable as a symmetric Newton 
polygon and so hypothesis (a) is not satisfied.
For $p \equiv 3,4,5,9 \bmod 11$, it is known that this Newton polygon occurs on $\mathcal{M}_5$ by
\cite[Theorem~5.4]{LMPT1}.

\item When $\xi=ss^5$: then hypothesis (b) of Proposition~\ref{Pnunonempty} is not satisfied.
For $p \equiv 7 \bmod 8$ and $p \gg 0$, there exists a supersingular smooth curve of genus $5$ by \cite[Theorem~1.2]{LMPT2}. 
\end{enumerate}
\end{example}

\bibliographystyle{amsalpha}
\bibliography{supersingular}

\end{document}